\documentclass[leqno,11pt,a4paper]{amsart}
\usepackage{amsthm}
\usepackage{cite}
\usepackage{enumitem}
\usepackage{hyperref}
\usepackage{graphicx}
\usepackage{amssymb}
\usepackage{xcolor}
\hypersetup{
  colorlinks,
  linkcolor={blue!50!black},
  citecolor={blue!50!black},
  urlcolor={blue!80!black}
}
\usepackage{multirow}
\usepackage[all]{xy}
\usepackage{subfigure}
\usepackage{tikz}
\usepackage{booktabs}
\usepackage{arydshln}
\usepackage{mleftright}

\setlist[enumerate]{labelsep=*, leftmargin=1.5pc}
\setlist[enumerate]{label=\normalfont(\roman*), ref=\roman*}
\usepackage[margin=2.7cm]{geometry}
\graphicspath{{images/}}
\theoremstyle{plain}
\newtheorem{thm}{Theorem}[section]
\newtheorem{pro}[thm]{Proposition}
\newtheorem{lem}[thm]{Lemma}

\newtheorem{conjecture}{Conjecture}
\theoremstyle{definition}
\newtheorem{dfn}[thm]{Definition}
\newtheorem{rem}[thm]{Remark}
\newtheorem{eg}[thm]{Example}

\DeclareMathOperator{\rk}{rk}

\DeclareMathOperator{\GL}{{GL}}

\DeclareMathOperator{\Spec}{Spec}

\DeclareMathOperator{\Newt}{Newt}
\DeclareMathOperator{\Ann}{Ann}

\DeclareMathOperator{\mut}{mut}
\DeclareMathOperator{\Arr}{Arr}
\DeclareMathOperator{\Cone}{Cone}
\DeclareMathOperator{\Gr}{Gr}

\newcommand{\bmax}[1]{\operatorname{max}\mleft({#1}\mright)}
\newcommand{\cB}{\mathcal{B}}

\newcommand{\cA}{\mathcal{A}}
\newcommand{\cE}{\mathcal{E}}

\newcommand{\cC}{\mathcal{C}}

\newcommand{\QQ}{{\mathbb{Q}}}

\newcommand{\PP}{{\mathbb{P}}}

\newcommand{\ZZ}{{\mathbb{Z}}}

\newcommand{\CC}{\mathbb{C}}

\newcommand{\bs}{\mathbf{s}}

\newcommand{\proofsection}[1]{%
	\vspace{0.2em}%
	\noindent\underline{{#1}:}%
	\ }

\begin{document}
\author[T.\,Prince]{Thomas Prince}
\address{Mathematical Institute\\University of Oxford\\Woodstock Road\\Oxford\\OX2 6GG\\UK}
\email{thomas.prince@magd.ox.ac.uk}

\title{Polygons of Finite Mutation Type}
\maketitle
\begin{abstract}
	We classify Fano polygons with finite mutation class. This classification exploits a correspondence between Fano polygons and cluster algebras, refining the notion of singularity content due to Akhtar and Kasprzyk. We also introduce examples of cluster algebras associated to Fano polytopes in dimensions greater than two.
	
\end{abstract}

\section{Introduction}
\label{sec:introduction}

The notion of combinatorial, or polytope, mutation was introduced by Akhtar--Coates--Galkin--Kasprzyk~\cite{ACGK} to describe mirror partners to Fano manifolds. Following Givental \cite{Givental:Equivariant_GW,Giv95,Giv98}, Kontsevich~\cite{K98}, and Hori--Vafa~\cite{Hori--Vafa}, the mirror partner to a Fano manifold consists of a complex manifold together with a holomorphic function, the \emph{superpotential}. If this mirror manifold contains a complex torus we can write down a collection of volume preserving birational maps of this complex torus which preserve the regularity of the superpotential. We call these rational maps (algebraic) mutations, following \cite{ACGK} and work of Galkin--Usnich~\cite{Galkin--Usnich}. Combinatorial mutation is the operation induced on the Newton polyhedra of the restriction of the superpotential to such tori.

All the polytopes we consider are \emph{Fano}, that is, polytopes which contain the origin in the interior and such that its vertices are primitive lattice vectors. In joint work \cite{KNP15} with Kasprzyk and Nill we showed that, in dimension two, the notion of polytope mutation is compatible with the construction of a quiver and cluster algebras one can associate to each Fano polygon. 

The idea of associating a polygon with a quiver -- or toric diagram -- has a reasonably long history, particularly in the physics literature. In that setting the polygon describes a toric Calabi--Yau singularity and the quiver is used to describe the matter content of a gauge theory arising on a stack of D$3$-branes probing the toric Calabi--Yau singularity, for example \cite{FHM+,HKPW,BP06,HV07,FHH01,LV98,AH97} for a selection of the literature on this subject. The construction of a quiver (and cluster algebra) from a polygon has also been used by Gross--Hacking--Keel \cite{GHK2} in the study of associated log Calabi--Yau varieties, and to study the derived category of the toric variety, or the associated local toric Calabi--Yau as pursued, for example, in \cite{BS10,H04,MR04,HP11,P12}. In each setting the basic construction is the same, and we recall the version relevant to our applications in \S\ref{sec:edge_mutations}.

Our main result, \textbf{Theorem~\ref{thm:finite_type}}, is a classification of the mutation classes of polygons which contain only finitely many polygons. This parallels a finite type result of Mandel~\cite{Mandel14}, for rank two cluster varieties. In particular we see that finite mutation classes of polygons fall into four types $A_1^n$, for $n \in \ZZ_{\geq 0}$, $A_2$, $A_3$, and $D_4$.

There is a close connection between mutation classes of Fano polygons and $\QQ$-Gorenstein deformations of the corresponding toric varieties which is described in detail in \cite{A+}. Following these ideas we predict the existence of a finite type parameter space for these deformations, together with a boundary stratification such that each zero stratum corresponds to a polygon in the given mutation class, and the $1$-strata corresponds to the mutation families constructed by Ilten~\cite{Ilten12}.

While our main result applies in dimension two, we note that polytope mutation is defined in all dimensions, and the construction of a quiver and cluster algebra we provides applies to `compatible collection' of mutations in any dimension, see Definition~\ref{def:compatible_collection}. This definition is, unfortunately, less well behaved in dimensions greater than two, but we provide an example indicating that polytope mutation can detect known examples cluster structures appearing on linear sections of Grassmannians of planes. We expect this to extend to a wide variety of other cluster structures found in Fano manifolds and their mirror manifolds.

\subsection*{Acknowledgements}
We thank Alexander Kasprzyk for his insights on polytope mutation, and our many conversations about quivers. The author is supported by a Fellowship by Examination at Magdalen College, Oxford. This work was undertaken while the author was a graduate student at Imperial College London.
\section{Quivers and cluster algebras}
\label{sec:clusters}

We devote this section to fixing the various conventions and notation, as well as recalling the basic definitions. We recall the definition of cluster algebra, and in order to address both geometric and combinatorial applications we shall adapt our treatment from the work of Fomin--Zelevinsky~\cite{FZ00}, and the work of Fock--Goncharov ~\cite{FG09} and Gross--Hacking--Keel~\cite{GHK2}. We first fix the following data:
\begin{itemize}
	\item $N$, a fixed rank $n$ lattice with skew-symmetric form  $\{-,-\}\colon N \times N \rightarrow \ZZ$.
	\item A saturated sublattice $N_{uf} \subseteq N$, the \emph{unfrozen} sublattice.
	\item An index set $I$, $|I| = \rk(N)$ together with a subset $I_{uf} \subseteq I$ such that $|I_{uf}| = \rk(N_{uf})$. For later convenience we shall define $m := |I_{uf}|$.
\end{itemize}

\begin{rem}
	The requirement that the form is integral is not necessary, but is sufficiently general for our applications and simplifies the exposition considerably.
\end{rem}

\begin{dfn}\label{def:seed}
	A \emph{(labelled) seed} is a pair $\bs = (\cE,C)$, where:
	\begin{itemize}
		\item $\cE$ is a basis of $N$ indexed by $I$, such that $\cE|_{I_{uf}}$ is a basis for $N_{uf}$.
		\item $C$ is a transcendence basis of $\mathcal{F}$, the field of rational functions in $n$ independent variables over $\QQ(x_i : i \in I \setminus I_{uf})$, referred to as a \emph{cluster}.
	\end{itemize}
\end{dfn}

\begin{rem}
	The basis $\cE$ is what the authors of \cite{FG09,GHK2} refer to as \emph{seed data}. Since we have fixed the lattice $N$ and skew-symmetric form $\{-,-\}$ the variables $x_i$ can be identified with coordinate functions on the \emph{seed torus} $T_N$.
\end{rem}

\begin{dfn}\label{def:seed_mutation}
	Given a seed $\bs = (\cE,C)$ with $\cE = \{e_1, \ldots , e_n\}$ and $C = \{x_1,\ldots,x_n\}$, the \emph{$j$th mutation} of $(\cE,C)$ is the seed $(\cE',C')$, where $\cE' = \{e_1',\ldots,e_n'\}$ and $C' = \{x'_1,\ldots,x'_n\}$ are defined by:
	\[
	e'_k = \begin{cases}
	-e_j, & \text{if $k = j$} \\
	e_k + \bmax{b_{kj},0}e_j, & \text{otherwise}
	\end{cases}
	\]
	\noindent where $b_{kl} = \{e_k,e_l\}$ and is often referred to as the \emph{exchange matrix},
	\begin{align}\label{eq:exchange_relation}
	x_k' = x_k\text{ if }k \ne j,&&
	\text{ and }&&
	x_jx'_j = \!\!\prod_{\stackrel{\scriptstyle k\text{ such that}}{b_{jk} > 0}}\!\!{x^{b_{jk}}_k} + \!\!\prod_{\stackrel{\scriptstyle l\text{ such that}}{b_{jl} < 0}}\!\!{x^{b_{lj}}_l}.
	\end{align}
\end{dfn}

\begin{dfn}
	\label{dfn:cluster_algebra}
	A \emph{cluster algebra} is the subalgebra of $\mathcal{F}$ generated by the cluster variables appearing in the union of all clusters obtained by mutation from a given seed.
\end{dfn}

Given a skew-symmetric $n \times n$ matrix $B$, we let $\cA(B)$ denote the cluster algebra associated to $B$; this is a subalgebra of $\QQ(x_1,\ldots,x_n)$.

\begin{rem}
	Definition~\ref{dfn:cluster_algebra} is really a special case of the definition of a cluster algebra, a class referred to as the \emph{skew-symmetric cluster algebras of geometric type}. In the general case the form $\{-,-\}$ need only be \emph{skew-symmetrizable}. One consequence of the skew-symmetry of the form $\{-,-\}$ is the identification of each exchange matrix with a \emph{quiver} $Q$. One may assign this quiver in the obvious way, assigning a vertex to each basis element of $N$, and $b_{ij}$ arrows $v_i \rightarrow v_j$, oriented according to the sign of $b_{ij}$. Having divided the vertex set into frozen vertices and unfrozen ones one can replace the basis $\cE$ with $Q$. There is a well-known notion of quiver mutation, going back to Fomin--Zelevinsky~\cite{FZ00}, generalising the reflection functors of Bernstein--Gelfand--Ponomarev~\cite{BGP73}. Mutating a seed in a skew-symmetric cluster algebra induces a corresponding mutation of the associated quiver.
\end{rem}

\begin{dfn}\label{def:quiver_mutation}
	Given a quiver $Q$ and a vertex $v$ of $Q$, the \emph{mutation of $Q$ at $v$} is the quiver $\mut(Q,v)$ obtained from $Q$ by:
	\begin{enumerate}
		\item\label{item:quiver_mutation_add_shortcuts}
		adding, for each subquiver $v_1 \to v \to v_2$, an arrow from $v_1$ to $v_2$;
		\item\label{item:quiver_mutation_delete_2_cycles}
		deleting a maximal set of disjoint two-cycles;
		\item\label{item:quiver_mutation_reverse_arrows}
		reversing all arrows incident to $v$.
	\end{enumerate}
	The resulting quiver is well-defined up to isomorphism, regardless of the choice of two-cycles in~\eqref{item:quiver_mutation_delete_2_cycles}.
\end{dfn}

Since we shall refer to quivers frequently we shall make the following conventions.	Given a quiver $Q$, we define

\begin{itemize}
	\item $Q_0$ to be the set of vertices of $Q$.
	\item $\Arr(v_i,v_j)$ to be the set of arrows from $v_i \in Q_0$ to $v_j \in Q_0$.
	\item $b_{ij}$ to be the cardinality of $\Arr(v_i,v_j)$, with sign indicating orientation.
\end{itemize}
We shall always assume $Q$ has no vertex-loops or 2-cycles.

Given a seed $\bs$ we shall also fix notation for the dual basis $\cE^\star$ of $M := \hom(N,\ZZ)$ and for each $i \in I$, set $v_i := \{e_i,-\} \in M$. We now define the  $\mathcal{A}$ and $\mathcal{X}$ cluster \emph{varieties} defined by Fock--Goncharov \cite{FG09}. Toward this, observe to a seed $\bs$ we can associate a pair of tori

\begin{align*}
\mathcal{X}_\bs = T_M && \mathcal{A}_\bs = T_N.
\end{align*}
The dual pair of bases for the respective lattices define identifications of these tori with split tori,
\begin{align*}
\mathcal{X}_\bs \longrightarrow \mathbb{G}^n_m, && \mathcal{A}_\bs \longrightarrow \mathbb{G}^n_m.
\end{align*}
We also associate the following birational maps to each seed,
\begin{align*}
\mu^\star_kz^n = z^n(1+z^{e_k})^{-\{n,e_k\}} && \mu^\star_kz^m = z^m(1+z^{v_k})^{\langle e_k,m \rangle}.
\end{align*}
Pulling these birational maps back along the identifications with the split torus given by the seed, the birational map $\mu_k \colon \mathcal{A}_\bs \dashrightarrow \mathcal{A}_{\mu_k(\bs)}$ is given by the exchange relation \eqref{eq:exchange_relation}. That is, this birational map is the coordinate-free expression of the exchange relation once we identify the standard coordinates on $T_N$ with the cluster variables $x_i \in C$ (including the frozen variables $x_{n+1} \cdots, x_m$). We obtain schemes $\mathcal{X}$ and $\mathcal{A}$ by gluing the seed tori $\mathcal{A}_s$ and~$\mathcal{X}_s$ along the birational maps defined by the mutations $\mu_k$. For more details and related results we refer to \cite{FG09,GHK2}.

We conclude this section by recalling the classifications of cluster algebras of \emph{finite type} and \emph{finite mutation type}.

\begin{dfn}
	A cluster algebra is said to be of \emph{finite type} if it contains finitely many clusters.
\end{dfn}

Given an undirected graph $G$ we say that a quiver $Q$ is an \emph{orientation} of $G$ if it has the same set of vertices and for each edge of $G$ there is precisely one arrow between the respective vertices. Given a simply-laced Dynkin diagram $D$ we say that $Q$ is of \emph{type $D$} if it is an orientation of the underlying graph of $D$.

\begin{thm}[\!\!\cite{FZ03}]
	\label{thm:finite_cluster_algebras}
	There is a canonical bijection between the Cartan matrices of
	finite type and cluster algebras of finite type. Under this bijection, a Cartan matrix $A$
	of finite type corresponds to the cluster algebra $\cA(B)$, where $B$ is an arbitrary skew-symmetrizable
	matrix with Cartan companion equal to $A$.
\end{thm}

Theorem~\ref{thm:finite_cluster_algebras} describes skew-symmetric cluster algebras with finitely many clusters. We can ask instead for the weaker condition that only finitely many \emph{quivers} appear associated to seeds of the cluster algebra. This is the notion of \emph{finite mutation type} cluster algebra, for which a  classification is also known.

\begin{thm}[{\cite[Theorem~$6.1$]{FST09}}]
	\label{thm:finite_mut_type}
	Given a quiver $Q$ with finite mutation class, its adjacency matrix $b_{ij}$ is the adjacency matrix of a triangulation of a bordered surface or is mutation equivalent to one of eleven exceptional types.
\end{thm}

\begin{figure}
	\makebox[\textwidth][c]{\includegraphics{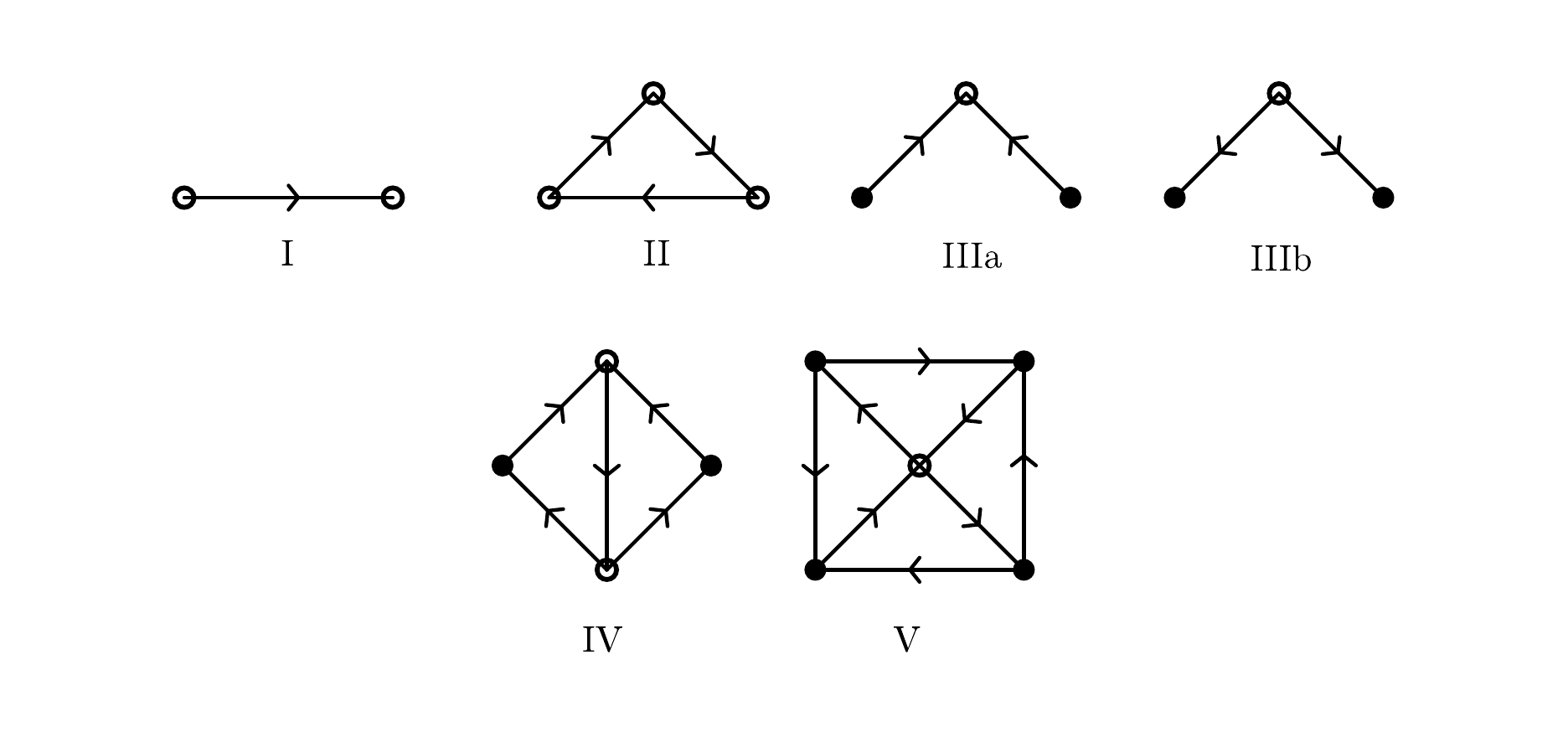}}
	\caption{the blocks of a \emph{block decomposition}}
	\label{fig:blocks}
\end{figure}

The class of quivers coming from triangulations of surfaces is well-studied and we make use of a combinatorial characterisation of this class of quivers via \emph{block decomposition}. A quiver $Q$ is said to admit a \emph{block decomposition} if it may be assembled from the $6$ \emph{blocks} shown in Figure~\ref{fig:blocks} by identifying the vertices of quivers shown with unfilled circles, the \emph{outlets}. More precisely, we choose a partial matching of the combined set of outlets such that no outlet is matched to a vertex of the same block, including itself. We form $Q$ by gluing the quiver along these vertices and cancelling any two cycles formed by this process. See \cite[Definition~$13.1$]{FST07} for further discussion and examples of this definition.

\begin{thm}[{\cite[Theorem~$13.3$]{FST07}}]
	\label{thm:blocks}
	A quiver $Q$ given by the adjacency matrix of a triangulation of a surface is mutation equivalent to a quiver which admits a \emph{block decomposition}
\end{thm}

\section{Mutations of Polytopes}
\label{sec:edge_mutations}

In two dimensions all combinatorial mutations are `tropicalisations' of cluster mutations. While this ceases to be true in higher dimensions there is a natural class of combinatorial mutations, the \emph{edge mutations} which do appear in this way. In terms of the definition of combinatorial mutation given in \cite{ACGK}, edge mutations are those which have one-dimensional \emph{factor}. In particular each edge mutation is obtained by studying the effect of the following birational maps -- an \emph{algebraic mutation}~\cite{ACGK} -- on the Newton polyhedra of certain Laurent polynomials. Throughout this section $N$ denotes an $n$-dimensional lattice (not necessarily related to the definition of a cluster algebra). We recall that, working over $\CC$, if $M$ is the lattice dual to $N$, the torus $T_M$ is defined to be $\Spec(\CC[N])$.

\begin{dfn}
	\label{def:alg_mutation}
	Given an element $w \in M$, the \emph{weight vector}, and $f \in \Ann(w)$, the \emph{factor}, define a birational map $\phi_{w,f} \colon T_M \dashrightarrow T_M$ sending
	\[
	z^n \mapsto z^n(1+z^f)^{\langle w,n\rangle}.
	\]
	Given a Laurent polynomial $W \in \CC[N]$ such that $\phi^\star_{w,f}(W) \in \CC[N]$ say that $W$ is mutable with weight vector $w$ and factor $f$. 
\end{dfn}

\begin{dfn}[Cf.\@ {\cite[p$12$]{ACGK}}]
	\label{def:mutation}
	Fix a Fano polytope $P \subset N_\QQ$ and its dual $P^\circ \subset M_\QQ$, a weight vector $w \in M$, and factor $f \in \Ann(w)$. Define a piecewise linear map $T_{w,f}\colon M_\QQ \rightarrow M_\QQ$ by setting
	\[
	T_{w,f} \colon m \mapsto m + \max(0,\langle m,f \rangle)w
	\]
	If $T_{w,f}(P^\circ)$ is a convex polytope then we say $P$ admits the mutation $(w,f)$ and that $P$ mutates to $(T_{w,f}(P^\circ))^\circ$.
\end{dfn}

\begin{rem}
	This definition of mutation is really a `dual characterisation' of \cite[Definition~$5$]{ACGK}, which encodes how the Newton polytope of a Laurent polynomial changes under algebraic mutations.
\end{rem}

\begin{rem}
	In \cite{ACGK} the authors show that the result of applying a mutation to a Fano polytope produces another Fano polytope, so the last dualization in Definition~\ref{def:mutation} is well-defined.
\end{rem}

\begin{pro}
	\label{pro:alg_to_combinatorial}
	Given $w \in M$,~$f \in \Ann(w)$ and a mutable Laurent polynomial $W \in \CC[N]$ we have the following identity;
	\[
	\Newt\left(\phi^\star_{w,f} W\right)^\circ = T_{w,f}\left(\Newt(W)^\circ\right).
	\]
\end{pro}
\begin{proof}
	The notion of combinatorial mutation is compatible with the mutation $W$ by construction. The interpretation of a combinatorial mutation as a piecewise linear map is made in the proof of Proposition $4$ in \cite{ACGK}.
\end{proof}

\begin{dfn}
	\label{def:compatible_collection}
	We define \emph{mutation data} to be elements $(w,f) \in M\oplus N$ such that $w$ and $f$ are primitive, and $f \in \Ann(w)$. A set of mutation data $\{(w_i,f_i) \in M \oplus N : i \in I\}$ for a finite index set $I$ is called \emph{compatible} if 
	\[
	\langle w_i,f_j \rangle = -\langle w_j,f_i\rangle.
	\]
\end{dfn}

\begin{rem}
	If $\dim N = 2$ any finite collection edge mutation data since $\langle w_i,f_j\rangle = w^i\wedge w^j$, for a choice of identification $\bigwedge^2M \cong \ZZ$.
\end{rem}

\begin{dfn}
	To a compatible collection of mutation data $\cE$ we define a quiver $Q_{\cE}$ as follows,
	\begin{enumerate}
		\item The vertex set of $Q_{\cE}$ is $\cE$.
		\item Between two vertices $(w_i,f_i)$ and $(w_j,f_j)$ there are $\langle w_i,f_j\rangle$ arrows, with sign indicating orientation.
	\end{enumerate}
\end{dfn}
Observe that, as $\langle w_i,f_j\rangle$ is skew-symmetric, the quiver $Q_\cE$ contains no loops or two cycles. Note that we can use this definition to assign a \emph{cluster algebra} to a compatible collection of mutation data. We define a rule governing how compatible collections of mutations themselves mutate.

\begin{dfn}
	\label{def:mutate_seeds}
	Given a compatible collection of mutation data $\cE$, let $L$ be the sublattice of $M \oplus N$ generated by the elements of $\cE$, and let $\{(w_i,f_i),(w_j,f_j)\} := \langle w_i,f_j\rangle$ define a skew-symmetric form on $L$. Fixing a pair $E_k = (w_k,f_k) \in \cE$ we \emph{mutate} $\cE$ to a new collection $\cE_k$ as follows:
	\begin{itemize}
		\item $E_k \mapsto -E_k$;
		\item $E_i \mapsto E_i - \max(\{E_i,E_k\},0)E_k$, if $i \neq k$.
	\end{itemize}
\end{dfn}

This formula is identical to the mutation of seed data given in~\cite{FG09}; a connection we now make precise. Fix a compatible collection of mutations $\cE$ and define a skew-symmetric form $[-,-]$ on $\ZZ^{\cE}$ defined by setting $[e_i,e_j] := \{\theta(e_i),\theta(e_j)\}$, where $\theta\colon \ZZ^\cE \to M\oplus N$ is defined by sending $e_i \mapsto (w_i,f_i)$. The following Lemma follows immediately by comparison of the formulae for mutating seed data in a cluster algebra with Definition~\ref{def:mutate_seeds}.

\begin{lem}
	\label{lem:fix_cluster} 
	The operations of mutation given in Definition~\ref{def:mutate_seeds}, and of mutation of the seeds defined above, are intertwined by $\theta$.
\end{lem}

\begin{rem}
	\label{rem:bad_fact}
	In dimensions higher than two a compatible collection of mutation data which defines a set of combinatorial mutations of a given polytope can transform by mutation to a compatible collection of mutation data which does not define a set of combinatorial mutations of the transformed polytope. In particular the piecewise linear maps may fail to preserve convexity. This appears to be a important obstruction to generalising the two-dimensional theory of mutations to higher dimensional polytopes. 
\end{rem}

\begin{pro}
	Given seed data $\cE$ such that $Q_{\cE}$ is a directed simply-laced Dynkin diagram the number of polytopes obtained by successive edge mutation is bounded by the numbers of seeds in the cluster algebra determined by $Q_{\cE}$. If $Q_\cE$ is of type $A_n$ this bound is the Catalan number $C_n$.
\end{pro}

In fact, compatible collections of mutations appear whenever we have a cluster algebra with skew-symmetric exchange matrix.

\begin{pro}
	\label{pro:cluster_algebra_polygon}
	Every compatible collection of mutations determines and is determined by a skew-symmetric cluster algebra without frozen variables, together with a subspace $V$ of the kernel of the skew-symmetric form $\{-,-\}$ defined by the exchange matrix.
\end{pro}
\begin{proof}
	Fix a skew-symmetric cluster algebra without frozen variables and a nominated subspace $V \subset \ker \{-,-\}$. Recall that a seed defines a basis $e_i$ of a lattice, which we denote $\widetilde{N}$. Define $M := \widetilde{N}/V$ and let $p \colon \widetilde{N} \rightarrow M$ be the canonical projection. The map $\theta\colon \widetilde{N} \rightarrow M \oplus\hom(M,\ZZ)$ defined by $\theta \colon n \mapsto ( p(n),\{n,-\})$ defines a compatible collection of mutation data with weight vectors in the lattice $M$.
\end{proof}

Note that $N$ and $M$ play dual roles to those in \cite{FG09}, and we insist throughout that $P \subset N_\QQ$. This exchange of roles explains the odd definition of $M$ in the proof of Proposition~\ref{pro:cluster_algebra_polygon}. To compare the birational maps associated to the two notions of mutations let $\bs$ be a seed of the cluster algebra determined by a compatible collection of mutation data, and let $\cE$ be the compatible collection corresponding to $\bs$. Fix an element $E_k = (w_k,f_k) \in \cE$ and consider the following diagram,
\begin{equation}
\label{eq:commuting_mutations}
\xymatrix{
	\mathcal{A}_{\bs} \ar@{-->}^{\mu_k}[rr] \ar^{p}[d] & & \mathcal{A}_{\mu_k(\bs)} \ar^{p}[d] \\
	T_{M} \ar@{-->}_{\phi_{(w_k,f_k)}}[rr] & & T_{M},
}
\end{equation}

\begin{pro}
	\label{prop:seed_mutations}
	The diagram shown in \eqref{eq:commuting_mutations} commutes
\end{pro}
\begin{proof}
	This is an exercise in writing out the definitions of the respective mutations, see~\cite[Section~3]{KNP15}.
\end{proof}

\begin{eg}
	The del~Pezzo surface of degree $5$ admits a toric degeneration to a toric surface $Z$ with a pair of $A_1$ singularities. Given a three-dimensional linear section $X$ of the Grassmannian $\Gr(2,5)$ $X$ admits a toric degeneration to the projective cone over $Z$. The fan determined by this toric threefold is formed by taking the cones over the faces of the reflexive polytope with PALP id $245$.
	
	In Figure~\ref{fig:B5_mutations} we show a pentagon of polytopes obtained by successively mutating the polytope shown in the top-right with respect to the mutation data
	\[
	\cE := \{(w_1,f_1),(w_2,f_2)\}
	\] 
	where,
	\begin{align*}
		w_1 := (-1,0,0), && f_1 := (0,1,1)^T, \\
		w_2 := (0,0,-1), && f_2 := (-1,0,0)^T. \\
	\end{align*}

	We recall that there is an $A_2$ cluster structure on co-ordinate ring of the Grassmannian, and a toric degeneration of $\Gr(2,5)$ for each cluster chart in the dual Grassmannian \cite{RW15,RW17}. We expect that cluster structures in the mirror to a Fano variety to be detected by such compatible collections of mutations.
	
	Note that the polytopes we show in Figure~\ref{fig:B5_mutations} are not dual to Fano polytopes. However, recalling that $B_5$ has Fano index $2$, we can obtain a reflexive polytope by dilating each of the polytopes shown in Figure~\ref{fig:B5_mutations} by a factor of two, and translating.
	
	\begin{figure}
		\includegraphics{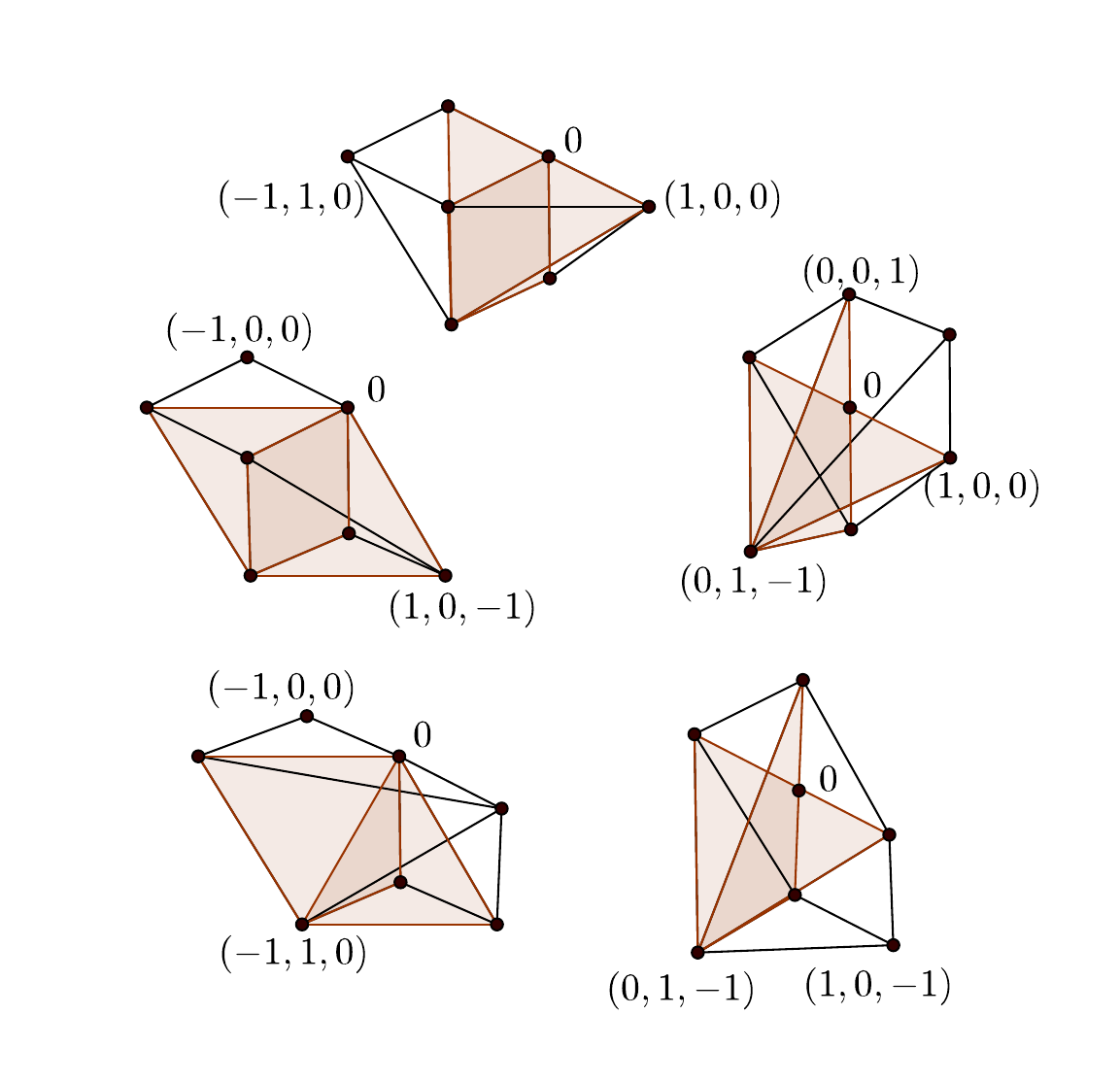}
		\caption{Pentagon of edge mutations among toric degenerations of $B_5$.}
		\label{fig:B5_mutations}
	\end{figure}
\end{eg}

In the two dimensional case, we can canonically define a maximal set of compatible mutations, making use of the notion of singularity content~\cite{AK14}.

\begin{dfn}[Cf.\@ {\cite[\S$1.2$]{KNP15}}]
	\label{def:polygon_seed}
	Given a Fano polygon $P \subset N_\QQ$ with singularity content $(n,\cB)$ and $m := |\cB|+n$, we define:
	\begin{itemize}
		\item An index set $I$ of size $m$, with a subset $I_{uf}$ of size $n$ and functions:
		\begin{align*}
		\phi_{uf} \colon I_{uf} \rightarrow \{\textrm{edges of $P$}\} && \phi_f \colon I \setminus I_{uf} \rightarrow \cB
		\end{align*}
		Here the fibre $\phi_{uf}^{-1}(E)$ has $m_E$ elements, where $m_E$ is the singularity content of $\Cone(E)$, and $\phi_f$ is a bijection.
		\item A lattice map $\rho \colon \ZZ^m \rightarrow M$ sending each basis element to the primitive, inward-pointing normal to the edge of $P$ defined by the cone given by the specified functions $\phi_{uf}$ and $\phi$.
		\item A form $\{e_i,e_j\} := \rho(e_i) \wedge \rho(e_j)$. Note that this is an integral skew-symmetric form.
	\end{itemize}
\end{dfn}

By \cite[Proposition~$3.17$]{KNP15} the construction of a quiver from a polytope provided by Definition~\ref{def:polygon_seed} intertwines polygon and quiver mutations. We let $(E_P,C_P)$ denote the cluster algebra associated to a Fano polygon, where $E_P$ is the standard basis $e_i$ of $\ZZ^n$, and $C_P$ is the standard transcendence basis of the field of rational functions in $n$ variables over $\QQ(x_i : i \in I\setminus I_{uf})$. We say a Fano polygon is of \emph{finite mutation type} if it is mutation equivalent to only finitely many Fano polygons.

\begin{conjecture}
	The cluster algebras $\mathcal{C}_P$ for Fano polygons $P$ together with a bijection between the set of frozen variables and $\mathcal{B}$ is a complete mutation invariant Fano polygons.
\end{conjecture}

\begin{eg}
	Consider the Fano polygon $P$ for $\PP^2$
	\begin{align*}
	\includegraphics[scale=0.75]{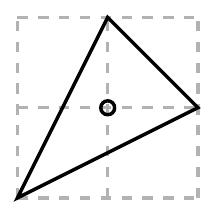}
	\end{align*}
	Computing the determinant of the inward-pointing normals we obtain the quiver $Q_P$
	\begin{align*}
	\xymatrix{
		& \bullet \ar^3[dr] & \\
		\bullet \ar^3[ur] &  & \bullet \ar^3[ll]
	}
	\end{align*}
	The mutations of this quiver are well-known, and the triple $(3a,3b,3c)$ of non-zero entries of the exchange matrix satisfy the Markov equation $a^2 + b^2 + c^2 = 3abc$. Indeed, as the polygon $P$ is mutated the corresponding toric surfaces are $\PP(a^2, b^2, c^2)$ for the same triples $(a,b,c)$. We see that in this case the mutations of the quivers exactly capture the mutations of the polygon.
\end{eg}

\begin{eg}
	\label{ex:special_surface}
	Consider the toric surface (using the notation for these surfaces appearing in \cite{A+}). $X_{5,5/3}$ associated with the Fano polygon shown below.
	\begin{align*}
	\includegraphics[scale=0.75]{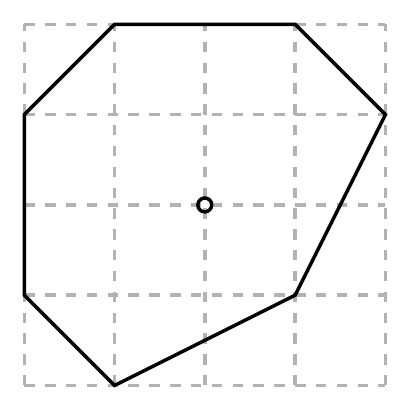}
	\end{align*}
	The quiver associated to this surface is simply the $A_2$ quiver,
	\[
	\xymatrix{
		\bullet \ar[r] & \bullet.
	}
	\]
	This example is important, both in this section, because it is an example of a \emph{finite type} polygon, and since a smoothing of this surface is given by $5$ Pfaffian equations, see \cite[\S$3.3$]{CH17}, a fact closely connected to the $A_2$ quiver we construct here.
\end{eg}

\section{Finite Type Classification}
\label{sec:classification}

We now make use of the classification of finite type and finite mutation type cluster algebras to establish the following result.

\begin{thm}
	\label{thm:finite_type}
	$P$ is of finite mutation type if and only if $Q_P$ is mutation equivalent to a quiver of type $(A_1)^n$, $A_2$, $A_3$, or $D_4$.
\end{thm}

\begin{rem}
	The types referred to in Theorem~\ref{thm:finite_type} may also be referred to as type $I_n$, $II$, $III$, and $IV$ respectively; in analogy with Kodaira's monodromy matrices. The relationship between these matrices, log Calabi--Yau manifolds, and monodromy in certain integral affine manifolds is explored by Mandel in~\cite{Mandel14}.
\end{rem}

We first make two straightforward observations. First we note that the cluster algebra $\mathcal{C}_P$ induces a sequence of surjections:

\begin{equation}
\label{eq:tower}
\xymatrix{
	\{ \textrm{Clusters of $\mathcal{C}_P$} \} \ar[d] \\
	\{ \textrm{Polygons mutation equivalent to $P$} \} \ar[d] \\
	\{ \textrm{Quivers mutation equivalent to $Q_P$} \}.
}
\end{equation}

The first vertical arrow follows from the fact that algebraic mutations determine combinatorial mutations, the second from Lemma~\ref{lem:fix_cluster}. For example, using this tower of surjections in the case of a type $A_2$ cluster algebra, we can immediately state the following result.

\begin{pro}\label{pro:pentagon}
	If a Fano polygon $P$ has singularity content $(2,\mathcal{B})$ and the primitive inward-pointing normal vectors of the two edges corresponding to the unfrozen variables of $\cC_P$ form a basis of the lattice $M$, then the mutation-equivalence class of $P$ has at most five members.
\end{pro}

\begin{proof}
	The quiver associated to $P$ is precisely an orientation of the $A_2$ quiver. The cluster algebra $\mathcal{C}_P$ is well-known and its cluster exchange graph forms a pentagon. Note however that the \emph{quiver} mutation graph is trivial, as the $A_2$ quiver mutates only to itself.
\end{proof}

 Proposition~\ref{prop:seed_mutations} implies that the mutation class of $P$ has at most five elements. Note that we do not have a non-trivial lower bound: there is only one polygon in mutation equivalent to the polygon described in Example~\ref{ex:special_surface} up to $\GL(2,\ZZ)$ equivalence. Next observe that the sequence of surjections shown in \eqref{eq:tower} immediately implies that
\[
\textrm{$\mathcal{C}_P$ finite type} \Rightarrow \textrm{$P$ finite mutation type} \Rightarrow \textrm{$\mathcal{C}_P$ finite mutation type}.
\]

\begin{lem}
	\label{pro:no_kronecker}
	Given a Fano polygon $P$ of finite mutation type, $Q_P$ does not contain a Kronecker subquiver
	\[
	Q_k := 
	\{\xymatrix{
		v_1 \ar^k[r] & v_2
	}\},
	\]
	where $k>1$ is the number of arrows from $v_1$ to $v_2$.
\end{lem}

\begin{rem}
	This result is expected from results on the corresponding cluster algebra. The Kronecker quiver defines a rank $2$ cluster algebra which is known not to be of finite type when $k > 1$. Given that $P$ is the Newton polygon of a superpotential which is itself a combination of cluster monomials, we expect the polygon $P$ to grow as we mutate.
\end{rem}
	
\begin{proof}[Proof of Lemma~\ref{pro:no_kronecker}]

	Assume there is a $Q_k$ subquiver of $Q_P$, with vertices $v_1$,~$v_2$ corresponding to edges $E_1$~$E_2$ of $P$. We define $\rho \colon \ZZ^2 \to M$ by mapping the standard basis to the primitive inward normal vectors $w_i$ to $E_i$ for $i \in \{1,2\}$. Let $P' \subset \QQ^2$ be the image of $P$ under $\rho^\star$. The resulting polygon in $\QQ^2$ is shown schematically in Figure~\ref{fig:std_form}.
	
	\begin{figure}
		\makebox[\textwidth][c]{\includegraphics[scale = 0.75]{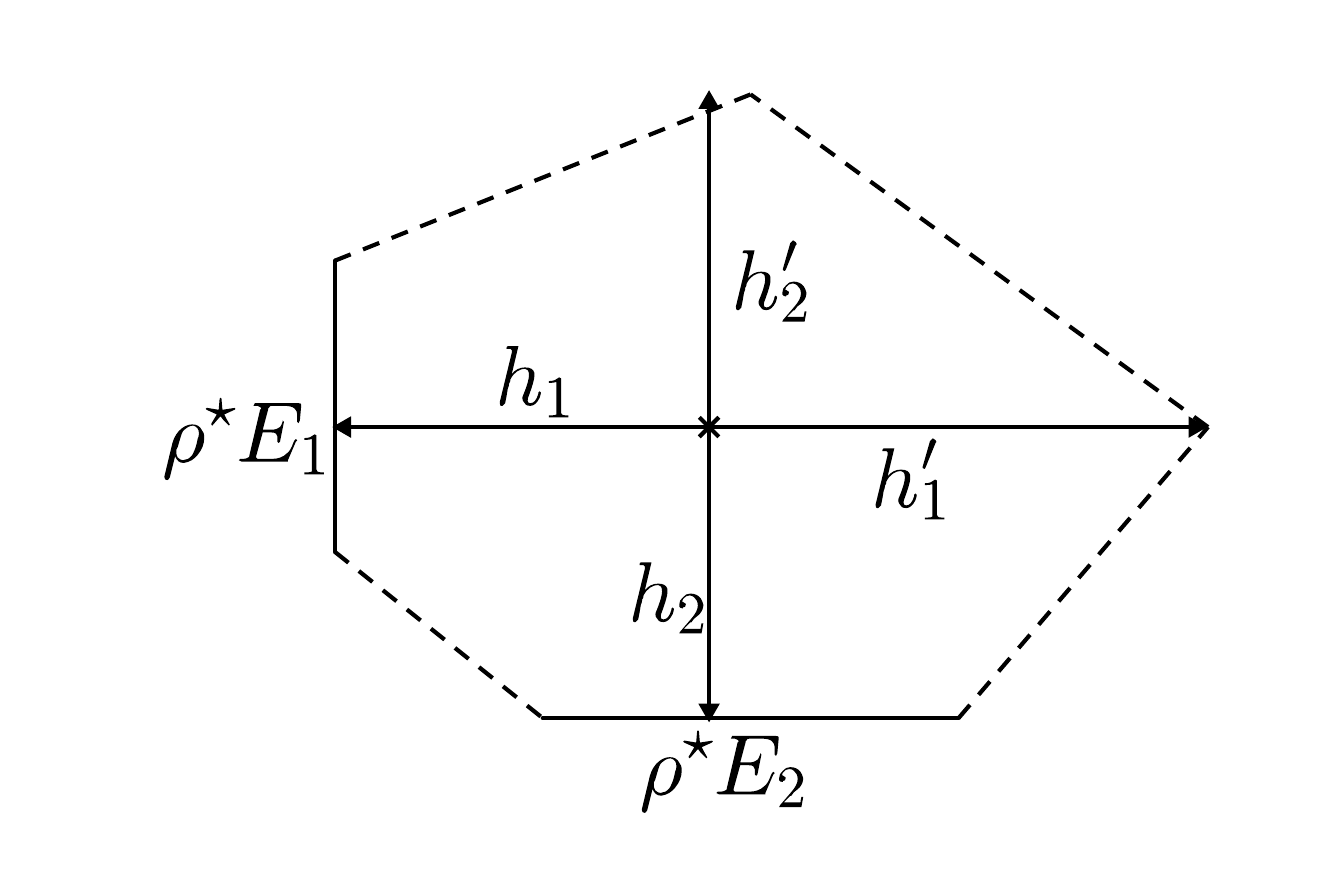}}%
		\caption{Schematic diagram of a polygon in standard form}
		\label{fig:std_form}
	\end{figure} 
	
	The \emph{local index} of each cone in $P$ is the integral height of the edge from the origin. Let $h_i$ denote the local indices of $E_i$ for $i \in \{1,2\}$. Note that, as $h_i = \langle e_i , \rho^\star v\rangle$ for any $v \in E_i$, $h_i$ is also the local index of $\rho^\star(E_i)$ in $P'$. Mutating at $v_1$ and $v_2$ we denote the new local indices,
	\[
	\xymatrix{
		(h_1,h_2') & (h_1,h_2) \ar[l] \ar[r] & (h_1',h_2).
	}
	\]
	We first show that $\rho^\star$ increases the lattice lengths of $E_i$ by a factor of $k := |w_1\wedge w_2|$ for each $i \in \{1,2\}$. Indeed, letting $\ell(E)$ denote the lattice length of an edge $E$, $v^i_1$ and $v^i_2$ denote the vertices of $E_i$, and $d_i := (v^i_1-v^i_2)/\ell(E_i)$ we have that:
	\begin{align*}
	\ell(\rho^\star(E_i)) &= \langle \ell(\rho^\star(E_i))e^\star_i,e_i  \rangle \\
	&=  \langle \rho^\star(v^i_1) - \rho^\star(v^i_2),e_i  \rangle\\
	&=  \langle v^i_1 -v^i_2,\rho(e_i)  \rangle \\
	&= \langle \ell(E_i)d_i,w_i  \rangle \\
	&= \pm\ell(E_i)(w_1\wedge w_2),
	\end{align*}
	where signs and orientations are chosen such that $\ell(E)$ is always positive. Studying Figure~\ref{fig:std_form}, we observe that:
	\begin{align*}
		h_1' \geq kh_2 - h_1 && h_2' \geq kh_1-h_2,
	\end{align*}
	Consider the case $k \geq 3$, and assume without loss of generality that $h_2 \geq h_1$. We have that $h_1' \geq 3h_2 - h_1 \geq 2h_2 \geq 2h_1$. Thus in this case the values in the pair $(h_1,h_2)$ grow (at least) exponentially with mutation, and in particular take infinitely many values.
	
	Next consider the case $k=2$. The inequalities above become,
	\begin{align*}
		h_1' \geq 2h_2 - h_1 && h_2' \geq 2h_1-h_2,
	\end{align*}
	\noindent and we are again free to assume that $h_2 \geq h_1$. Thus $h_1' \geq 2h_2-h_1 \geq h_1$, and if $h_2 > h_1$, $h_1' \geq 2h_2-h_1 > h_2$. Thus, assuming $h_1 \neq h_2$, one can generate an infinite increasing sequence of local indices. The only remaining case is if $h := h_1 = h_2 = h'_1 = h'_2$. To eliminate this possibility observe that, since $k=2$, the edges $\rho^\star(E_1)$,~$\rho^\star(E_2)$ must meet in a vertex with coordinates $(-h,-h)$ (indeed, assuming this does not hold, a mutation returns us to the previous case and one of the above inequalities is strict). Note that the sublattice $\rho^\star(N)$ is determined by the fact that $\rho^\star$ doubles the edge lengths of $E_1$ and $E_2$. The lattice vectors $(a,a)$ are in this sublattice for all $a \in \ZZ$. Thus, by primitivity of the vertices in $P$, $h=1$.  Since the origin is in the interior of $P$, mutating in one of $v_1$ or $v_2$ returns us to the previous case.
\end{proof}

\begin{rem}
	Proposition~\ref{pro:no_kronecker} implies all the quivers that we consider from now on are directed graphs. Hence we refer to vertices as \emph{adjacent} if they are adjacent in the underlying graph.
\end{rem}

As well as the non-existence of Kronecker quivers in $Q_P$ for finite mutation type polygons $P$, we use heavy use of a connectedness result for quivers $Q_P$ which follows immediately from the definition of $Q_P$ via determinants in the plane; or equivalently from the fact the exchange matrix has rank $2$.
	
\begin{lem}
	\label{lem:transitivity}
	Given a Fano polygon $P$ and vertices $v_1$,~$v_2$,~$v_3$ of $Q_P$ such that $v_i$ and $v_{i+1}$ are not adjacent for $i=1$,~$2$, then $v_1$ and $v_3$ are not adjacent.
\end{lem}

\begin{proof}[Proof of Theorem~\ref{thm:finite_type}]
	
	By Lemma~\ref{lem:transitivity}, if $Q_P$ is not connected, $Q_P \cong A_1^n$ for some $n$. Similarly, if $Q_P$ is of type $A$ or $D$, then it must be one of $A_2$,~$A_3$ or $D_4$. Thus we only need to show that there are is no Fano polygon $P$ of finite mutation type such that $\mathcal{C}_P$ is not of finite-type. However $\mathcal{C}_P$ is of finite mutation type, and we use the classification described in Theorems~\ref{thm:finite_mut_type} and~\ref{thm:blocks}, following \cite{FST07,FST09}.	In fact, using Lemma~\ref{lem:transitivity}, none of the eleven exceptional types can occur as $Q_P$ for a Fano polygon $P$. Hence we can restrict to quivers which admit a \emph{block decomposition} and work case-by-case.

	We claim that every quiver $Q_P$ associated to a Fano polygon $P$ which admits a block decomposition is either mutation equivalent to an orientation of a simply-laced Dynkin diagram or to a quiver which contains a subquiver $Q_k$ for $k>1$. We assume for contradiction that $Q_P$ is the quiver associated to a Fano polygon $P$ of finite-type which is not mutation equivalent to a simply laced Dynkin diagram.
	
	\proofsection{Block V}
	
	\begin{figure}
		\makebox[\textwidth][c]{\includegraphics[scale=1.5]{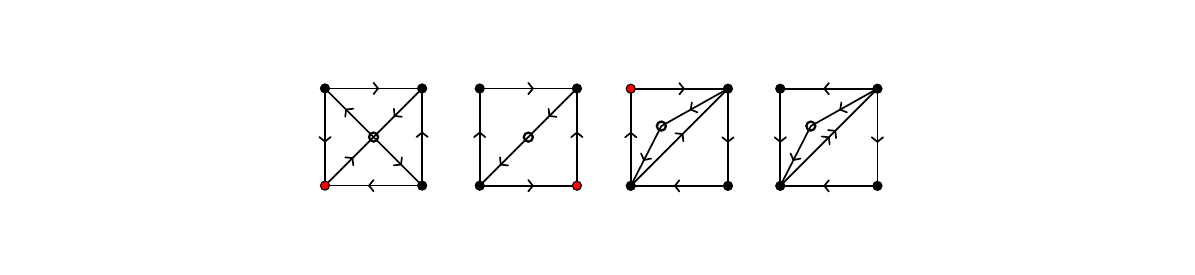}}
		\caption{Mutations of block V}
		\label{fig:block_V_mutations}
	\end{figure}
	
	First observe that, since only one vertex of the block V is an outlet, the V block quiver is a subquiver of any quiver which contains the V block in its decomposition. However this quiver mutates to a quiver with a $Q_2$ subquiver as shown in Figure~\ref{fig:block_V_mutations}. Therefore block V never appears in a decomposition of a quiver $Q_P$. For later use we shall fix the following intermediate quiver, $\text{V}'$, shown in Figure~\ref{fig:Intermediate_block}.
	\vspace{10mm}

	\begin{figure}
		\makebox[\textwidth][c]{\includegraphics[scale = 1.5]{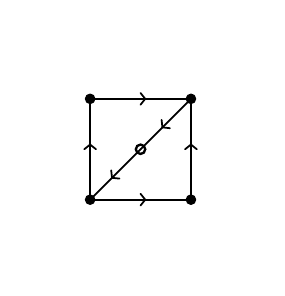}}%
		\caption{Quiver $\text{V}'$}
		\label{fig:Intermediate_block}
	\end{figure}
	
	\proofsection{Blocks IIIa and IIIb}
	
	Assume there is a type III block (a or b) connected to a quiver $Q'$ at a vertex $v$. If there is a vertex $v'$ of $Q'$ such that $v$ and $v'$ are not adjacent, the quiver violates Lemma~\ref{lem:transitivity}. In particular the vertex set of $Q'$ must be the vertex set of a single block. In particular, using the previous part, $Q'$ has at most four vertices. Case by case study shows that only the $A_3$ and $D_4$ types appear.
	
	\proofsection{Block IV}
	
	Consider the case of a decomposition only using type IV blocks. Note that the type IV block is itself of type D$4$. Consider attaching two type IV blocks. If the blocks are attached at a single outlet the resulting quiver contradicts Lemma~\ref{lem:transitivity}. In fact it is easy to see that it is impossible to add additional type IV blocks to meet this condition. If both pairs of outlets are matched there are two possible quivers depending on the relative orientations of the arrow between the outlets, one orientation produces a $Q_2$ subquiver automatically, the other produces a quiver containing the quiver $\text{V}'$ as a subquiver. Thus, for a type IV block to appear in a decomposition of $Q_P$ it must include a type I or II block.
	
	Now consider decompositions using type I and II blocks as well as type IV blocks. First note there must be exactly one IV block (assuming there is at least one). Indeed, if type IV blocks are not connected using both vertices, a non-outlet vertex of a IV block is not adjacent to some outlet, and some non-outlet vertex of a (different) IV block. However outlets and non-outlets of a type IV block are always adjacent, violating Lemma~\ref{lem:transitivity}.

	\begin{figure}[ht]
		\centering
		\subfigure[Attaching I blocks to a IV block]{%
			\includegraphics{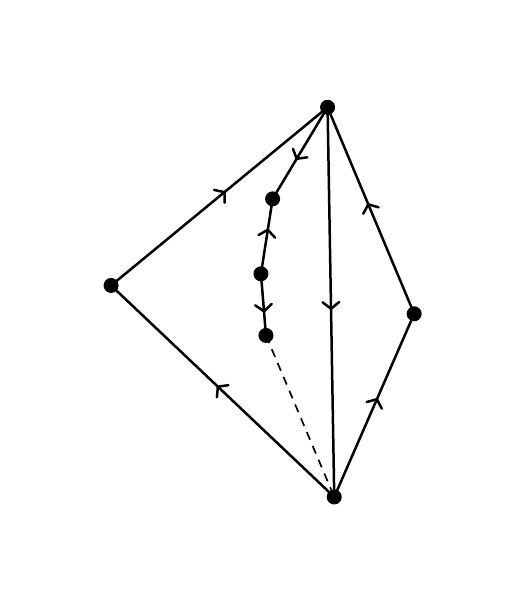}
		}
		\quad
		\subfigure[Attaching II blocks to a IV block]{%
			\includegraphics{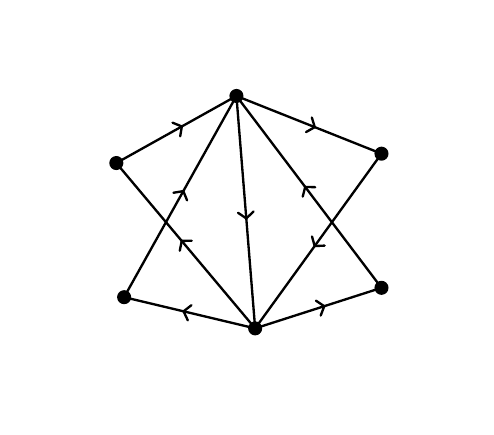}
		}
		\label{fig:IV_blocks}
	\end{figure}
	
	Thus we must attach I and II blocks to a single type IV block. By Lemma~\ref{lem:transitivity} the vertex set of the final quiver must be equal to the vertex set obtained by attaching a single block to each outlet of the IV block. Considering these cases in turn, we note first that attaching a type I block to cancel the arrow between the two outlets produces a quiver mutation equivalent to $D_4$ and therefore eliminated. For chains type I blocks of length two, if a $3$-cycle is produced, a mutation in the vertex between the type I blocks produces the $\text{V}'$ quiver. If not, the same mutation produces a $Q_2$ subquiver.
	
	Attaching a type II block along two outlets of the type IV block recovers the $\text{V}'$ or $Q_2$ subquiver cases we have already seen. Attaching type II blocks to a single outlet each we observe that every new vertex must be adjacent to both outlets of the IV block. Hence the only case without a $Q_2$ subquiver is shown on the right of Figure~\ref{fig:IV_blocks}, however this quiver mutates to one with a $Q_2$ subquiver. Attaching further type II blocks any quiver we obtain must contradict Lemma~\ref{lem:transitivity}.
	
	\proofsection{Blocks I and II}
	
	From what we have shown above, the block decomposition of $Q_P$ consists only of type I and type II blocks. Any connected quiver with a block decomposition into type I blocks is a path (with possibly changing orientations), which possibly closes up into a cycle. The only cases not violating Lemma~\ref{lem:transitivity} are mutation equivalent to orientations of simply laced Dynkin diagrams.

	\begin{figure}
			\includegraphics[scale=0.75]{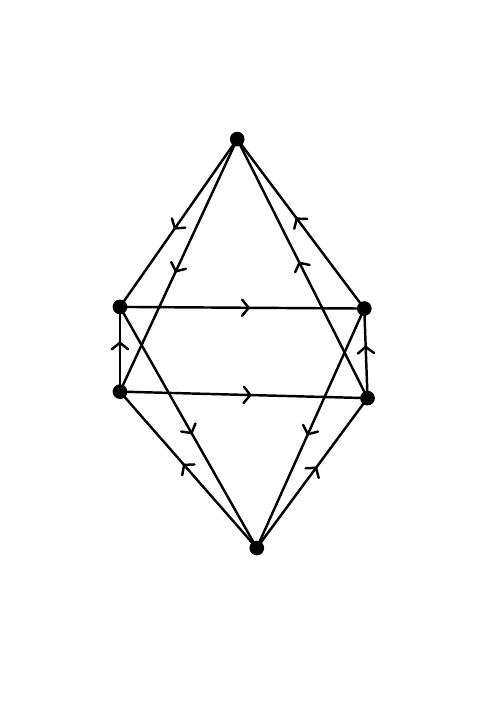}
			\caption{Octahedron of type II blocks.}
			\label{fig:Octahedral_quiver}
	\end{figure}

	For decompositions of $Q_P$ with type I and II blocks we divide the proof into cases indexed by the number of type II blocks. For a single type II block, we can attach a type I block to two outlets and in this way reduce to the type III case. Attaching each type I block to a type II block in at most one outlet, we use the fact that every new vertex must be adjacent to at least two of the vertices of the type II block. Thus we can obtain only two undirected graphs -- the underlying graph of a type IV block or an orientation of a tetrahedron, these cases can easily be eliminated. For example, there is no orientation of the tetrahedron making every cycle oriented; hence after a single mutation we obtain a quiver violating Proposition~\ref{pro:no_kronecker}.
	
	Consider the case of a pair of type II blocks. If these have disjoint vertex sets, each outlet of a type II block cannot be adjacent to \emph{two} of the outlets of the other type II block. Thus we must cancel the arrow between these two outlets with a type I block. However this creates a pair of $1$-valent non-outlet vertices which can be eliminated similarly to the type III case. At the other extreme, if we attach along all three outlets, we produce two easy cases. Attaching along a pair of outlets we generate either a $Q_2$ subquiver or a $4$-cycle. Considering the $4$-cycle with two outlets $v_1$ and $v_2$ (on non-adjacent corners) to meet the conditions of Lemma~\ref{lem:transitivity} any vertex adjacent to one of $v_1$ or $v_2$ must be adjacent to the other. Moreover, if the resulting quiver contains an arrow between $v_1$ and $v_2$, a mutation at one of the non-outlet vertices gives a $Q_2$ subquiver. Given a vertex $v$ adjacent to $v_1$ and $v_2$, if this defines a path between them, mutating at this node and a non-outlet in the four cycle produces a $Q_2$ subquiver. If $v$ does not lie on a path between $v_1$ and $v_2$ then mutating at both outlets produces a $Q_2$ subquiver.
	
	Attaching the type II blocks at a single outlet, the four arrows incident to this vertex are now fixed, so any new vertex must be adjacent to each of the remaining four outlets by Lemma~\ref{lem:transitivity}. However this cannot be achieved with type I blocks.

	Attaching more than two type II blocks together, we can eliminate the case where two are connected to form a 4-cycle as above. Since we can easily eliminate the case that two type II blocks meet in three outlets, we assume that each type II block meets every other in at most one outlet. Some pair of type II blocks must be attached in an outlet (otherwise we can argue as in the case of type II block separated by type I blocks). Thus, since every new vertex must be adjacent to all four outlets formed by attaching two type II blocks, all possible quivers can be represented as an octahedron with some orientation, see Figure~\ref{fig:Octahedral_quiver}.
	
	Considering an orientation of the octahedron; if any triangular face does not form a cycle we can mutate to form a $Q_2$ subquiver. Assuming every triangle is a cycle, and possibly mutating, the vertices adjacent to the `top' of the octahedron form a type V block subquiver. Following the same reasoning as for the type V block case (although note that the type V block is not part of a block decomposition here) these cases can be eliminated.
\end{proof}

\bibliographystyle{plain}
\bibliography{bibliography}
\end{document}